\documentclass[a4paper,10pt]{amsart}

\usepackage{amsmath}
\usepackage{amssymb}
\newcommand{\C}{\mathbb{C}}

\newcommand{\N}{\mathbb{N}}

\newcommand{\e}{\varepsilon}

\DeclareMathOperator{\Hol}{{Hol}}

\DeclareMathOperator{\codim}{{codim}}

\usepackage{amsthm}
\theoremstyle{definition}

\newtheorem*{thm}{Theorem}

\newtheorem*{remark1}{Remark  1}
\newtheorem*{remark2}{Remark  2}
\newtheorem*{question1}{Question   1}
\newtheorem*{question2}{Question   2}
\newtheorem*{question3}{Question   3}

\theoremstyle{plain}

\newtheorem*{corollary1}{Corallary 1    (Fundamental Theorem of Algebra) }
\newtheorem*{corollary2}{Corallary 2   (Gelfand's Theorem)}

\begin{document}

\title[FUNDAMENTAL THEOREM OF ALGEBRA]{A BANACH ALGEBRAIC VERSION OF THE FUNDAMENTAL THEOREM OF ALGEBRA}
\author{Ali Taghavi}

\address{Faculty of Mathematics and Computer Science,  Damghan  University,  Damghan,  Iran.}
\email{taghavi@du.ac.ir}


\begin{abstract}
For a monic polynomial $p(z)$ with coefficients in a unital complex Banach algebra, we prove that there exist a complex number $z$ such that $p(z)$ is not invertible.
\end{abstract}

\maketitle

\section*{Introduction}
\noindent
The purpose of this paper is to give a proof for a generalization of the fundamental theorem of algebra, briefly
FTA,  as follows:
For a monic polynomial $p(z)$ with coefficients in a unital complex Banach algebra $A$, there exist a complex number $z$ such that $p(z)$ is not an invertible element of A. Putting $A=\mathbb{C}$, it give's us a new proof of the fundamental theorem of algebra. The method of proof of our main theorem is essentially based on the theory of Fredholm operators on Banach spaces and the stability of Fredholm index under small perturbation. The part of Fredholm operator theory that will be used in this paper, is valid for both complex  and real Banach spaces, equivalently. This would shows that this new proof of FTA is selfcontained and is not a circular argument.\\
For an element $a$ in a unital complex Banach algebra, we apply our main theorem to the polynomial $p(z)=z-a$, then we obtain a new proof for the fact that the spectrum of $a$ is not empty. In the last section of the paper, we give some questions that are somehow related to the main theorem or its method of proof.

\section{Preliminaries}

Let $A$ be a unital complex Banach algebra. By $\ell^{1}(A)$,  we mean the space of all sequence $(a_{i})_{i\in\mathbb{N}}$ , $a_{i} \in A$,  such that
 $\sum_{i=0}^\infty |a_i|<\infty $. $\ell^{1}(A)$ is a Banach  space, with the norm $|(a_{i})|=\sum_{i=0}^\infty |a_i|$ see,  \cite [p.51] {KADRIN}. A sequence $(a_{i})$ is called eventually zero, if there exist a natural number $N$ such that for all $i \geq N$,  $a_{i}=0$. It can be easily shown that the space of all  eventually zero sequences is a dense subspace of $\ell^{1}(A)$.\\
 By $\ell^2$ we mean the Hilbert  space of all square summable  sequence $(z_{i})$ of complex numbers. $B(\ell^2)$ is the space of bounded operators on $\ell^2$ and $\mathcal{K}$ is the space of compact operators on  $B(\ell^2)$. The Calkin algebra is the quotient  algebra $B(\ell^2)/\mathcal{K}$. A norm preserving operator on a Banach space $X$ is called an isometry. Using
an straightforward  application of the triangle inequality, one can easily shows that for an isometry $T\in B(X)$, there exist a neighborhood $U$ of $T$, with respect to operator norm, such that any $S\in U$ satisfies $|S(x)|\geq k|x|$, for all $x\in X$ and some positive constant k. Then any such operator $S$ must be injective. Furthermore,  if $S(x_{n})$ is a Cauchy sequence, then $(x_{n})$ is a cauchy sequence, too.\\
Let $X$ be a Banach space, a bounded operator $T\in B(X)$ is called a Fredholm operator if  $\dim \ker T<\infty$, $\codim Rang (T) = \dim(X/Rang(T)) <\infty$ and Rang(T) is closed in $X$. For such $T$, the Fredholm index of $T$ is $Ind(T)=dim \ker T-\codim Rang(T)$.
The (essential) spectrum of an operator $T$ is the set of  all complex numbers $\lambda$ such that $T- \lambda I$ is not (fredholm) invertible.
For every $m \in \N$, the operator
\[
 \left\{ \!\!
  \begin{array}{l}
    S_m:\ell^1(A) \to \ell^1(A) \\[-2ex]
    (a_0,a_1,a_2, \ldots) \longmapsto (\overbrace{0,0,\ldots,0}^{m\text{ times}},a_0,a_1, \ldots)
  \end{array}\right.
\]
is called the \emph{$m$-shift} operator on $\ell^1(A)$.\\
 Obviously $S_{m}$ is a fredholm operator with $Ind(S_{m})=-m$. It is well known that the spectrum of the shift operator $S_{1}$ is the unit disc $\mathbb{D}=\{z\in\mathbb{C} :|z|\leq 1\}$   and its essential spectrum is the unit circle $S^{1}=\{z\in\mathbb{C} :|z|=1\}$. Moreover for each $\lambda$ with $|\lambda |<1$, $ind(S_{1}-\lambda)=-1$,
 see \cite[Chapter 11, Example 4.10]{ACFA} , and for each $\lambda$ with $|\lambda |>1$, \;   $Ind(S_{1}- \lambda) =0$\; since $S_{1}-\lambda$ is an  invertible operator.  $S_{m}$ is also an isometry.\\

The space of  Fredholm operators is an open subset of $B(X)$ and the index is locally constant,  then it remains unchanged with small perturbation. This property is called "invariance of Fredholm index with small perturbation". For two Fredholm operators $T$ and $S$, $TS$ is also Fredholm and $Ind(TS)=Ind(T)+ind(S)$. For a Hilbert space $H$, $T\in B(H)$ is fredholm if and only if the coset of  $T$ is an invertible element of the Calkin algebra $B(H)/\mathcal{K}(H)$. This implies that the composition of two (or more) commuting operators in $B(H)$ is Fredholm if and only if each of them is a  Fredholm operator.  Because in an algebra, the product of commuting elements is invertible if and only if each of them is invertible,  see ~\cite[p.271]{Rudin}.  For more information  about Fredholm operator theory see  \cite[Chapter 2]{CBS} and  \cite[Chapter 11]{ACFA}.\\
A map $f:\mathbb{C}\to A$ is called a holomorphic map if, for all $z_{0}\in \mathbb{C}$,
 $f'(z_{0})= \lim_{z \to z_{0}} \frac{f(z)-f(z_{0})}{z-z_{0}}$ exists. The set of all holomorphic maps from
 $\mathbb{C}$ to $A$ is denoted by $Hol(\mathbb{C}, A)$. Obviously $Hol(\mathbb{C}, A)$ is closed under addition and multiplication. Moreover assume  that $f\in Hol(\mathbb{C}, A)$ and $f(z)$ is an invertible element of $A$, for all $z$. Then the map $f(z)^{-1}$ is a holomorphic map from $\mathbb{C}$ to A. Because $\frac{f(z)^{-1}-f(z_{0})^{-1}}{z-z_{0}}=-f(z)^{-1}\frac{f(z)-f(z_{0})}{z-z_{0}}f(z_{0})^{-1}$. Taking a limit we obtain $(f^{-1})^{'} (z_{0})=-f(z_{0})^{-1}f'(z_{0})f(z_{0})^{-1}$.\\
 Let $f\in Hol(\mathbb{C}, A)$, then $f$ can be expanded as a convergent power series $f(z)=\sum_{n=0}^{\infty} a_{n}z^{n}$, $a_{i}\in A$, see  ~\cite[Theorem~3.3.1]{KADRIN}. Moreover  $\sum |a_{n}|<\infty$,  because $|a_{n}|=\frac{|a_{n}z^{n}|}{|z^{n}|}$, $a_{n}z^{n}$ is a bounded sequence, being  the common term of a convergent series  and $\sum1/|z^{n}|$ is a convergent geometric series, provided $|z|>1$. So a comparison test implies that $\sum |a_{n}|<\infty$.\\
 Then there is  a natural embedding of $Hol(\mathbb{C}, A)$ into
 $\ell^{1}(A)$, which sends $f(z)=\sum_{n=0}^{\infty} a_{n}z^{n}$ to $ (a_{n}) $. With this embedding, $Hol(\mathbb{C}, A)$  is considered as a dense subspace of $\ell^{1}(A)$, since it contains all polynomials and the space of all polynomials corresponds to the space of all eventually zero sequence, which is a dense subspace of $\ell^{1}(A)$.\\
 Note that the shift operator $S_{1}: \ell^{1}(A) \longmapsto \ell^{1}(A)$ restricts to multiplication operator on $Hol(\mathbb{C}, A)$ that sends $f(z)$ to $zf(z)$.

\section{Main Result}
In this section we prove the main theorem of the paper, which can be considered as a Banach algebraic analogy of FTA
\begin{thm}
  For every  monic polynomial $p(z)$ with coefficients in a unital complex Banach algebra, there exists a
  $z\in\mathbb{C}$ such that $p(z)$ is not invertible.
\end{thm}

\begin{proof}
  Assume that $p(z)=z^n+ a_{n-1}z^{n-1}+ \cdots +a_1z+a_0$  and
  for every $z\in \C$,  $p(z)$ is invertible. Then for all $z$ and all $\e\neq0$
  \[
   q_{\e}(z)=\e^n p(z/\e^n)= z^n +\e a_{n-1}z^{n-1} + \e^2 a_{n-2} z^{n-2} + \cdots + \e^n a_0,
  \]
  is invertible. Define an operator $M_{\e}:\Hol(\C) \to \Hol(\C)$ by multiplication
  $M_{\e}(f)= q_{\e}f$.
The operator $M_{\e}$ can be extended to
  a bounded linear operator $Q_{\e}:\ell^{1} (A) \to \ell^{1}(A)$ where  $Q_{\e}=S_{n}+\e a_{n-1}S_{n-1}+\cdots+\e^{n}a_{0}$.  Since $q_{\e}(z)$  is invertible, then $q_{\e}(z)^{-1}$ is a holomorphic map, so $M_{\e}$ is a surjective operator. We use this to prove that $Q_{\e}$ is also surjective.\\
  Let $b\in \ell^{1}(A)$, there is a sequence $b_{n}\in Hol(\mathbb{C}, A)$ that converges to b, since
  $Hol(\mathbb{C}, A)$ is dense in $\ell^{1}(A)$. Since $M_{\e}$ is surjective, there is a sequence $x_{n}\in Hol(\mathbb{C}, A)$ with $Q_{\e}(x_{n})=M_{\e}(x_{n})=b_{n}$. $Q_{\e}$ is a perturbation of the isometric operator $S_{n}$. Then by the argument that we stated in section $1$, $(x_{n})$ must be a Cauchy sequence, so it converges to a point $x\in \ell^{1} (A)$. This implies that $Q_{\e}(x)=b$, so $Q_{\e}$ is surjective. $Q_{\e}$ is also injective because it is a perturbation of an isometry then $Ind(Q_{\e})=0$. On the other hand   $Ind(Q_{\e})= Ind(S_{n})=-n$, by invariance of Fredholm index with small perturbation. This is a contradiction and the proof of the theorem is completed.

\end{proof}

The following two corollaries are immediate consequence of the above theorem:

\begin{corollary1}

A polynomial equation $P(z)=0$ with complex coefficients has at least one root.

\end{corollary1}

\begin{corollary2}
The spectrum of an element of a unital complex Banach algebra is not empty
\end{corollary2}

\begin{remark1}We explain that our new proof of the FTA is really a self contained proof. Namely, the FTA is not implicitly used in our proof.  The proof of the main theorem is based on Fredholm operator theory and invariance of Fredholm index with small perturbation. The part of this theory that we needed, has the same statement and formulation for both real and complex Banach spaces. The other object that we used in our proof is $Hol(\mathbb{C}, A)$. The properties of this space that we used, are really  independent of the FTA. So the FTA is not used in this new proof, implicitly.\\
There is a classical proof of the corollary 2, Gelfand's theorem, which use's the Liouville's theorem, or some related methods, see \cite[Theorem 10.13]{Rudin}.  Obviously the method of this classical proof can be applied to give another proof for the main theorem of this paper. In fact $p(z)^{-1}$ is a bounded holomorphic function defined on $\mathbb{C}$, so is  a constant map. But this classical proof is actually based on the theory of integration of Banach space
valued maps defined on complex numbers,  since the  Cauchy integral formula is used in the proof of Liouville's theorem.  So both proofs of our main theorem, the classical one and ours, use independently two different theory: the first is the integration theory which leads to Liouville's theorem and the second is Fredholm operator theory.

\end{remark1}

\section{Questions}
let $p(z)=z^n+ a_{n-1}z^{n-1}+ \cdots +a_1z+a_0$ be a monic polynomial with complex coefficients. We assign to $P(z)$, the operator $P=S_{n}+a_{n-1}S_{n-1}+\cdots+a_{0}$, as a bounded operator on $\ell^{2}$. We prove directly that $P$ is a Fredholm operator if and only if the polynomial $p(z)$ has no root on the unit circle $S^{1}$ in $\mathbb{C}$. In this case $Ind (P)=-k$,  where $k$ is the number of roots of $p(z)$ that lie in the interior of unit circle. However if we apply the essential spectrum on both side of the equation $P=p(S_{1})$,  we would obtain another argument for the first part of this statement.\\
 First assume that $P$ is a Fredholm operator. let $p(z)=\prod_{i=1}^{n}(z-r_{i})$, we factor $P$ in the form  $P=\prod_{i=1}^{n}(S_{1}-r_{i})$. $P$ is a composition of $n$ commuting operators $S_{1}-r_{i}$, then each $S_{1}-r_{i}$ must be Fredholm, so $|r_{i}|\neq 1$ for all $i$,\;    since the essential spectrum of the shift operator $S_{1}$ is the unit circle $S^{1}$. Assume that $|r_{i}|<1$ for $i\leq k$ and $|r_{i}|>1$ for $i>k$.  Then $ind(S_{1}-r_{i})=0$ \; for $i>k$ and  $Ind(S_{1}-r_{i})=-1$ \;for $i\leq k$, as we mentioned in section $1$ . Then   $Ind(P)=\sum_{i=1}^{n}Ind(S_{1}-r_{i})=-k$. The proof of the converse is trivial and is omitted. Now consider the finite dimensional vector space $W$ consisting all polynomials $p(z)$ with complex coefficients such that degree of $p$ is not greater than $n$. $W$ can be considered as a finite dimensional subvector space of $B(\ell^{2})$. Put F(W)=$\{T\in W :T$ is a Fredholm operator\}. The above computation shows that the index, as a function on F(W), is a bounded function. This simple observation leads us to the following question;
\begin{question1}
Let  $W$ be a finite dimensional subvector space of $B(\ell^{2})$. We denote by $F(W)$,  the set of all operators in $W$ that are Fredholm operators. Is the Fredholm index, a bounded function on $Fred(W)$.
\end{question1}

\begin{remark2}
The above question  can be interpreted as a question about the geometry of Fredholm operators. It actually asks: Can infinite number of connected components of $Fred(\ell^{2})$, intersect a finite dimensional subspace of $B(\ell^{2})$ ?
\end{remark2}
   A  complex algebra can not be necessarily considered as a complex Banach algebra. For example,  $C^{\infty} [0 \;\;1]$, the space of all smooth complex valued functions  on $ [0 \;\;1]$, admits no Banach algebra norm. For proof see \cite[p.51]{KAN}.  So it is natural to ask for a purely algebraic version of the main theorem. On the other hand, for a complex algebra $A$, the spectrum of an element may be empty. But if $\dim A$ is countable, then the spectrum is not empty, see \cite [Theorem 2.1.1]{CHG}. So the following question arises naturally :\\

\begin{question2}
Is the purely algebraic version of the main theorem true,\;  for countable dimensional algebras?
\end{question2}

In the main theorem, we can not drop the hypothesis that $p(z)$ is a monic polynomial, for example let a be a  quasinilpotent element of $A$, that is the spectrum of a is $\{0\}$, then $za+1$ is an invertible element, for all $z\in\mathbb{C}$. So it seems natural to ask the next question:

 \begin{question3}
 In the main theorem, can we drop the hypothesis of monicness of $p(z)$, provided that $A$ has no nontrivial quasinilpotent element?
  \end{question3}

\end{document}